\documentclass[12pt, a4paper]{article}

\usepackage{amsmath}
\usepackage{amsthm}
\usepackage{graphicx}
\usepackage[margin=1in]{geometry}
\usepackage{enumitem}

\newtheorem{theorem}{Theorem}[section]
\newlist{thmcases}{enumerate}{1}
\setlist[thmcases]{
  label=\textbf{\upshape Case~\thetheorem.\arabic*},
  leftmargin=*,
  ref={\thetheorem.\arabic*}}

\newtheorem{lemma}[theorem]{Lemma}
\newtheorem{definition}[theorem]{Definition}
\newtheorem{corollary}[theorem]{Corollary}

\title{On graphs representable by pattern-avoiding words}

\author{Yelena Mandelshtam}

\date{\today}

\begin{document}
\maketitle

\begin{abstract}
In this paper we study graphs defined by pattern-avoiding words. Word-representable graphs have been studied extensively following their introduction in $2000$ and are the subject of a book published by Kitaev in 2015. Recently there has been interest in studying graphs represented by pattern-avoiding words. In particular, in 2016, Gao, Kitaev, and Zhang investigated $132$-representable graphs, that is, word-representable graphs that can be represented by a word which avoids the pattern $132$.  They proved that all $132$- representable graphs are circle graphs and provided examples and properties of $132$-representable graphs. They posed several questions, some of which we answer in this paper.

One of our main results is that not all circle graphs are $132$-representable, thus proving that $132$-representable graphs are a proper subset of circle graphs, a question that was left open in the paper by Gao et al. We show that $123$-representable graphs are also a proper subset of circle graphs, and are different from $132$-representable graphs. We also study graphs represented by pattern-avoiding $2$-uniform words, that is, words in which every letter appears exactly twice.\end{abstract}

\pagebreak

\section{Introduction}

In this paper we study graphs defined by pattern-avoiding words. Word-representable graphs have been investigated extensively following their introduction in \cite{intropaper} and are the subject of the book \cite{book}. Recently there has been interest (see \cite{fill1}, \cite{fill2}, \cite{fill3}) in studying graphs represented by pattern-avoiding words. In particular, Gao, Kitaev, and Zhang studied $132$-representable graphs, that is, word-representable graphs which can be represented by a word which avoids the pattern $132$.  They proved that all $132$-representable graphs are circle graphs and provided examples and properties of $132$-representable graphs. They also posed several questions, some of which we answer in this paper.

One of the main results in this paper is that not all circle graphs are $132$-representable, thus proving that $132$-representable graphs are a strict subset of circle graphs, a question that was left open in \cite{mainpaper}. We show that $123$-representable graphs are also a proper subset of circle graphs, and are different from $132$-representable graphs. We also study graphs represented by pattern-avoiding $2$-uniform words, that is, words in which every letter appears exactly twice.

This paper is organized as follows. In Section \ref{prelims} we introduce important definitions, notation, and past results that will be used in the paper. In Section \ref{123graphs} we prove several results about $123$-representable graphs. In Section \ref{2unifgraphs} we discuss and prove some properties of graphs which can be represented by $2$-uniform pattern-avoiding words, and provide some examples of such graphs. In Section \ref{132graphs} we prove that not all circle graphs are $132$-representable, answering one of the questions posed in \cite{mainpaper}. Finally in Section \ref{opendir} we give some research directions.

Figure \ref{hierarchy} shows the hierarchy of graph classes, as established in this paper, with some examples of graphs fitting into each category.

\begin{figure}[htbp]
\centering
\includegraphics[width=\textwidth]{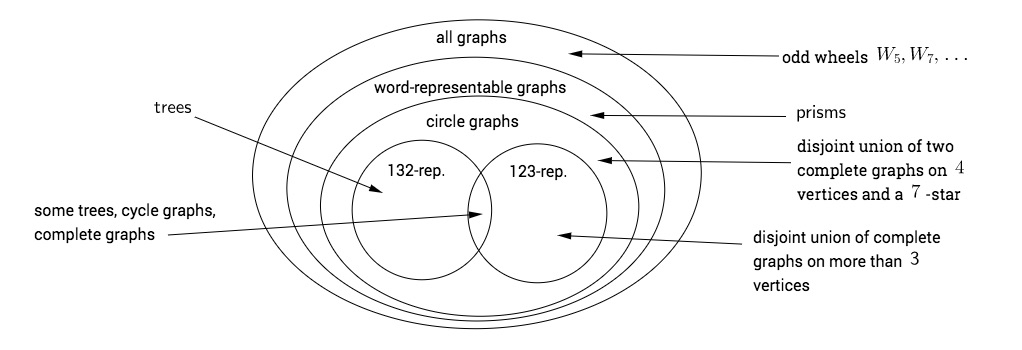}
\caption{\label{hierarchy} The place of $123$ and $132$-representable graphs in the hierarchy of graph classes.
}
\end{figure}

\section{Preliminaries and definitions} \label{prelims}

We will now introduce notation and definitions.

\begin{subsection}{Words and permutations}

Throughout this paper, $w$ refers to a word $w_1w_2 \cdots w_n$ over a totally ordered alphabet. 

\begin{definition}
A word $w$ is $k$-uniform if there are exactly $k$ copies of each letter in $w$.
\end{definition}

For example, the word $12432143$ is $2$-uniform, whereas $1232342$ is not.

\begin{definition} \label{alternateDef}
Two letters $x$ and $y$ \emph{alternate} in a word $w$ if there is an instance of $x$ between any two instances of $y$ and an instance of $y$ between any two instances of $x$.
\end{definition}

For example, in the word $abcbd$, the following pairs of letters are alternating: $(a, c), (a, d), \linebreak (b, c), (c, d).$

\begin{definition} \label{containingDef}
A word $w$ \emph{contains} the pattern $\tau = \tau_1 \cdots \tau_k$ if there are indices $1 \leq i_{a_1} < \dots < i_{a_k} \leq n$ such that $w_{i_{a_1}}, \dots , w_{i_{a_k}}$ is order-isomorphic to $\tau$
\end{definition}
In particular, $w$ contains the pattern $123$ if there is a strictly increasing substring of length $3$ in $w$. For example, $31247$ is a $123$-containing word.

\begin{definition} \label{avoidingDef}
A word $w$ \emph{avoids} a pattern if it does not contain it.
\end{definition}

Thus, the word $7546231$ is $123$-avoiding, whereas $7534621$ is not.

\end{subsection}

\begin{subsection}{Word-representable graphs}

In this paper all graphs are simple. The degree of a vertex is denoted by $d(v)$. 

\begin{definition}
A \emph{circle graph} $G = (V, E)$ is a graph whose vertices can be associated with chords of a circle such that two chords $a$ and $b$ intersect if and only if $ab \in E$.
\end{definition}

\begin{definition}
A graph $G = (V, E)$ is \emph{word-representable} if there exists a word $w$ over $V$ such that $x$ and $y$ alternate in $w$ if and only if $xy \in E$. Any such $w$ is said to represent $G$.
\end{definition}

\begin{figure}[htbp]
\centering
\includegraphics[width=0.4\textwidth]{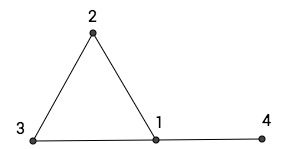}
\caption{\label{wordrepfig}An example of a word-representable graph. A $123$-avoiding word that represents this graph is $32414$.}
\end{figure}

\begin{definition}
Given a pattern $\tau$, a graph is $\tau$-representable if, possibly after relabeling the vertices of the graph, it can be represented by a word which avoids $\tau$.
\end{definition}
In particular, a graph is 123-representable if it can be represented by a 123-avoiding word. See Figure \ref{wordrepfig}.

Note that labeling is important when dealing with $\tau$-representable graphs, for a pattern $\tau$ as opposed to simply word-representable graphs. Figure \ref{labelingEx} shows the importance of a correct labeling. The graphs on the right and left are the same graph, but while the one on the left is $132$-representable with its current labeling, the one on the right is not.

\begin{figure}[htbp]
\centering
\includegraphics[width=0.7\textwidth]{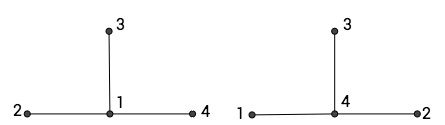}
\caption{\label{labelingEx}An example showing the importance of labeling correctly.}
\end{figure}

\end{subsection}

\begin{subsection}{Preliminaries}

We need the following results.

\begin{theorem}\emph{(\cite{mainpaper})}
Any $132$-representable graph is a circle graph.
\end{theorem}

\begin{theorem}\emph{(\cite{circlepaper})} \label{circletwo}
A graph $G$ is word-representable and has a representant with at most two copies of each letter if and only if $G$ is a circle graph.
\end{theorem}

\begin{lemma}\emph{(\cite{mainpaper})}\label{2uniflemma}
If $G_1, G_2, \dots , G_k$ are the connected components of a graph $G$ that can be $132$-represented by $2$-uniform words $w_1, w_2, \dots , w_k$, respectively, then $G$ is $132$-representable by a $2$-uniform word.
\end{lemma}

\end{subsection}

\section{123-representable graphs}\label{123graphs}

In this section we discuss properties of words representing $123$-representable graphs and properties of the graphs themselves.

We first present a simple, but useful generalization of Theorem 3.1 in \cite{mainpaper}. 
\begin{theorem} \label{easyprop}
Let $G$ be a word-representable graph, which can be represented by a word avoiding a pattern $\tau$ of length $k+1$. Let $x$ be a vertex in $G$ such that $d(x) \geq k$. Then, any word $w$ representing $G$ that avoids $\tau$ must contain no more than $k$ instances of $x$. 
\end{theorem}

\begin{proof} 
Since $d(x) \geq k$, there must be vertices $a_1, a_2, \dots ,a_k$ that are adjacent to $x$, and each of $a_1, \dots, a_k$ must be alternating with $x$ in $w$.

Now suppose there are at least $k+1$ copies of $x$ in $w$. Then there exists a subsequence $x_1w_1x_2 \cdots w_kx_{k+1}$ in $w$, where each $w_1 \cdots w_k$ must have a copy of each $a_1 \cdots a_k$. However, now every single possible permutation of $x, a_1, a_2, \dots , a_k$ can be found in the word, since we can get any permutation of the $a_i$'s by simply taking the first element from $w_1$, the second from $w_2$, etc. Then we can find an $x$ between any two elements of the permutation of the $a_i$s, so we have all possible permutations of $x, a_1, a_2, \dots , a_k$. Therefore, all possible patterns of length $k+1$ can be found in $w$. It follows that $x$ can appear in $w$ no more than $k$ times for $W$ to avoid a pattern of length $k+1$.
\end{proof}

Note that in the case of $123$-representability, Theorem \ref{easyprop} implies that in a graph $G$, if $x$ has degree at least $2$, then $x$ can appear at most two times in any $123$-avoiding word representant for $G$.

\begin{corollary} \label{easycor}
Let $w$ be a word-representant for a graph which avoids a pattern of length $k+1$. If some vertex $a$ adjacent to $x$ has degree at least $k$, then $x$ occurs at most $k+1$ times in $w$.
\end{corollary}
 
\begin{proof}
By Theorem \ref{easyprop}, $a$ occurs at most $k$ times in $w$. Since $x$ must alternate with $a$, there can be no more than $k + 1$ instances of $x$, for otherwise, $x$ and $a$ would not be alternating.
\end{proof}

Note that if $x$ is a vertex of degree $1$ in a graph $G$  and a vertex $a$ adjacent to $x$ has degree at least $2$, then $x$ occurs at most three times in any word that is a $123$-representant for $G$.

The following lemma will be very useful in the proof of Theorem \ref{shitcase}.

\begin{lemma} \label{switch}
If $w$ is a $123$-representant for $G = (V, E)$ and has a factor $ab$, with $a < b$, then $w'$, formed by switching $a$ and $b$, is also $123$-avoiding. Furthermore, $w'$ represents $G$ if $ab \not\in E$, and $a$ and $b$ do not alternate in $w'$.
\end{lemma}

\begin{proof}
Assume for the sake of contradiction that $w'$ is not $123$-avoiding, and thus contains an increasing subsequence of length $3$. Since the only difference between $w$ and $w'$ is the order of $a$ and $b$, this means that they both must be in this subsequence. However, they are in decreasing order, so this is impossible. Thus $w'$ is $123$-avoiding.

The only possible edges (or lack thereof) that could have been affected by the switch are edges incident with $a$ or $b$. However, the relative order of $a$ and $b$ is irrelevant when considering an edge only involving one of $a$ or $b$. Thus if $ab \not \in E$, and $a$ and $b$ do not alternate in $w'$, then $w'$ represents $G$.
\end{proof}

Note that Lemma \ref{switch} says that if we have a word $w = w_1a_1a_2 \cdots a_kxw_2$ that is $123$-avoiding, with all $a_i < x$ for $1\leq i \leq k$, then the word $w_1xa_1a_2 \cdots a_kw_2$ is also $123$-avoiding and represents almost the same graph as $w$, except with possible changes in the connections between $x$ and all of the $a_i$'s.

The following theorem will be the main ingredient in proving that all $123$-representable graphs are circle graphs.

\begin{theorem} \label{shitcase}
If a graph $G$ is $123$-representable, then there exists a $123$-avoiding word $w$ representing $G$ such that any letter in $w$ appears at most twice.
\end{theorem}

\begin{proof}
Let $w$ be a $123$-representant for $G = (V, E)$. We will show that if there is a letter $x$ which appears more than two times, we can create a new word that is still a $123$-representant of $G$ but has $x$ appearing only twice and does not change the frequency of the other letters. Then this process can be repeated until every letter appears at most twice.

We begin by noting that if a vertex in $V$ has degree at least two, then by Theorem \ref{easyprop}, the corresponding letter appears at most twice. Hence we need only consider vertices which have degree $1$ or $0$.

\begin{thmcases}
\item{$d(x) = 0$: \newline Let $w_1$  be the word obtained by deleting all instances of $x$ from $w$, and then re-label the graph so that the vertex previously labeled $x$ now has the largest value of any of the vertices, while preserving the relative order of the other vertices. We denote this value $x'$. The desired word is $x'x'w_1$. This word is still $123$-avoiding since $x'$ cannot participate in a $123$-pattern as it is at the beginning of the word. Furthermore, the word still represents $G$ because $x'$ cannot alternate with anything, and all other alternating pairs have not been affected.
}

\item{$d(x) = 1$ and the vertex $a$ connected to $x$ also has degree $1$: \newline This case means that the edge $xa$ is disconnected from the rest of the graph. Let $w_1$ be the word obtained from $w$ after deleting all instances of $a$ and $x$. Now we form the final word $x'a'x'a'w_1$, where $x' > a'$ and $a'$ is greater than all the values of the letters in $w_1$. The new word is still $123$-avoiding and it represents $G$ because $x'$ alternates with $a'$ and neither of them alternates with any other letter.
}

\item{$d(x) = 1$ and the vertex $a$ adjacent to $x$ has degree at least $2$: \newline By Corollary \ref{easycor}, $x$ cannot appear more than $3$ times in $w$. Then we have the word $$w = w_1x_1w_2a_1w_3x_2w_4a_2w_5x_3w_6.$$

Now consider the word $w_3$. If $w_3$ is nonempty, let $c$ be a letter in $w_3$. Then $5$ of the $6$ possible permutations of $axc$ appear in $w$, namely, $x_1a_1c$, $x_1ca_2$, $cx_2a_2$, $ca_2x_3$, and $a_1cx_2$. The only one that does not appear is  $axc$. Thus, for $w$ to be $123$-avoiding, $a < x < c$ must hold for any $c \in w_3$.

Similarly, if $w_4$ is nonempty, and it contains some $c$, then the only possible ordering of $a$, $x$, and $c$ is $c < x < a$, as $w$ contains $x_1a_1c$, $x_1ca_2$, $ca_2x_3$, $a_1x_2c$, and $a_1cx_3$. 

Because both inequalities cannot hold simultaneously, $w_3$ and $w_4$ cannot both be nonempty.

Let $b$ be another vertex adjacent to $a$. Then we know $b$ must be either in $w_3$ or $w_4$ since it alternates with $a$. There are six possible relative orders of $a$, $b$, and $x$:

\begin{enumerate}
\item{$x < a < b$: This is impossible since $x_1a_1b$ is a $123$-pattern.
}
\item{$x<b<a$: This is impossible since $x_1ba_2$ is a $123$-pattern.
}
\item{$b<a<x$: This is impossible since $ba_2x_3$ is a $123$-pattern.
}
\item{$a<b<x$: This is impossible since $a_1bx_3$ is a $123$-pattern.
}
\item{$a<x<b$: This is only possible if $b \in w_3$. Otherwise, if $b \in w_4$, then $a_1x_2b$ is a $123$-pattern.
}
\item{$b<x<a$: This is only possible if $b \in w_4$. Otherwise, if $b \in w_3$, then $bx_2a_2$ is a $123$-pattern.
}
\end{enumerate}

This gives us two subcases.

\begin{description}
\item[Subcase 1]{$a<x<b$ and $b \in w_3$ \newline This implies that $w_4$ is empty since $w_3$ is nonempty. This gives us $$w = w_1x_1w_2a_1w_3bw_4x_2a_2w_5x_3w_6.$$

If $w_5$ is nonempty, consider some $c \in w_5$. Note that $w$ contains the patterns $a_1cx_3$ and $a_1x_2c$. Then, since $w$ is $123$-avoiding, $c<a$. Thus every letter in $w_5$ is less than $a$ (vacuously true if $w_5$ is empty). Now let $$w' = w_1w_2a_1w_3bx_2a_2x_3w_5w_6.$$ By Lemma \ref{switch}, $w'$ is $123$-avoiding, and it only remains to make sure that $x$ shares an edge with $a$ and does not share an edge with any vertex in $w_5$ (since $d(x) = 1$). Since $a_2$ is the only letter between $x_2$ and $x_3$, $x$ can only alternate with $a$. In fact, it clearly does alternate with $a$, so we have successfully created a word $w'$ which only contains $x$ twice.
}

\item[Subcase 2]{$b < x < a$ and $b \in w_4$ \newline This implies that $w_3$ is empty since $w_4$ is nonempty. This gives us $$w = w_1x_1w_2a_1x_2bw_4a_2w_5x_3w_6.$$

If $w_2$ is nonempty, consider some $c \in w_2$. Note that $w$ contains the patterns $x_1ca_1$ and $cx_2a_2$. Thus, since $w$ is $123$-avoiding, $x < c$. Therefore every letter in $w_2$ is greater than $x$ (vacuously true if $w_2$ is empty). Now let $$w' = w_1w_2x_1a_1x_2w_3bw_4a_2w_5w_6.$$ By Lemma \ref{switch}, $w'$ is $123$-avoiding and it only remains to make sure that $x$ is only alternating with $a$. It is clear that $x$ alternates with $a$, and since the only letter between $x_1$ and $x_2$ is $a_1$, $x$ does not alternate with any other letters. Thus we have successfully created a word $w'$ which is a $123$-representant of $G$ and only contains $x$ twice.
}

\end{description}
In both cases we were able to create a new word which represents the same graph but only contains two copies of the letter $x$. 

}
\end{thmcases}
We have gone through all possible cases, and have shown that it is possible to create a $123$-representant for a graph $G$ which has no more than two copies of each letter, since we can repeat the process outlined above for every letter which appears more than two times. Therefore there exists a word for any $123$-representable graph which contains no more than two copies of each letter.
\end{proof}

The next corollary follows easily.
\begin{corollary}
Any $123$-representable graph is a circle graph.
\end{corollary}

\begin{proof}
By Theorem \ref{shitcase}, any $123$-representable graph can be represented by a $123$-avoiding word with at most two copies of each letter. From Theorem \ref{circletwo}, this implies that any $123$-representable graph is a circle graph.
\end{proof}

Finally, we prove that not all circle graphs are $123$-representable. We begin with a lemma.

\begin{lemma} \label{opporder}
Let $x$ be a vertex in a graph $G = (V, E)$ with $xa, xb \in E$. If $ab \notin E$, and $a$ and $b$ appear on both sides of $x$ in some word $w$ representing $G$ with at most two copies of each letter, then $a$ and $b$ appear in opposite orders on both sides of $x$.
\end{lemma}
\begin{proof}
Assume $a$ and $b$ appear in the same order (without loss of generality, $a$ followed by $b$) on each side of $x$. Then $a$ alternates with $b$, but since $ab \notin E$ this is impossible. Thus they must appear in opposite order.
\end{proof}

\begin{theorem}
The star $K_{1, 6}$ is not $123$-representable.
\end{theorem}

\begin{proof}
Consider the star $K_{1, 6}$ (Figure \ref{star}), and suppose $w$ is some $123$-representant for it. At most one of the vertices labeled $a$ through $f$ can appear once in the word, since if two appear only once then they alternate with each other, a contradiction. Thus, without loss of generality, $a, b, \dots, e$ appear twice in $w$. There are two cases.

\begin{thmcases}
\item{$x$ appears once. \newline It is sufficient to consider only $a$, $b$ and $c$. Without loss of generality, let $a < b < c$. Note that $b$ must appear after $c$ on the right of $x$, for otherwise $a_1b_2c_2$ forms a $123$-pattern. Furthermore, $a$ must come after $b$ on the left because otherwise $a_1b_1c_2$ would form a $123$-pattern. Now, by Lemma \ref{opporder} the only two possibilities for the order in which the letters appear in $w$ are $b_1a_1c_1xc_2a_2b_2$ and $b_1c_1a_1xa_2c_2b_2$.

We have $x<c$ for otherwise $a_1c_1x$ would be a $123$-pattern. Furthermore, $a < x$ for otherwise $xa_2b_2$ would form a $123$-pattern. However, if $a < x < c$, $a_1xc_2$ forms a $123$-pattern. Thus the star cannot be $123$-represented with only one instance of $x$.
} \label{caseone}
\item{$x$ appears twice. \newline We must have each $a, b, c, d, e$ appearing between $x_1$ and $x_2$. The second copies of $a, b, c, d, e$ can be either to the right of $x_2$ or to the left of $x_1$. By the Pigeon Hole Principle, there must be at least three letters to one side of the $x$'s. Without loss of generality, let $a$, $b$, and $c$ appear on both sides of $x_1$. Then we have reduced this argument to the same one examined in Case \ref{caseone}, so we are done.
}
\end{thmcases}

\begin{figure}[htbp]
\centering
\includegraphics[width=0.4\textwidth]{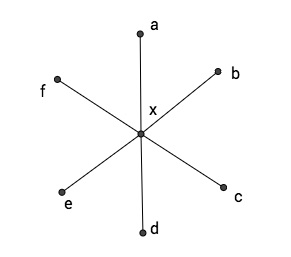}
\caption{\label{star}The star $K_{1, 6}$.}
\end{figure}

Therefore, the star $K_{1, 6}$ is not $123$-representable. 
\end{proof}

\begin{corollary}
Not all circle graphs are $123$-representable.
\end{corollary}
\begin{proof}
In the previous theorem we showed that the $7$-star is not $123$-representable. However, since it is a tree, it is a circle graph. Thus not all circle graphs are $123$-representable.
\end{proof}

\subsection{Examples of $123$-representable graphs}
In this section we give three families of $123$-representble graphs.

\begin{theorem}
Complete graphs are $123$-representable.
\end{theorem}
\begin{proof}
The complete graph on $n$ vertices can be represented by the word $n(n-1) \cdots 21$. Clearly this is $123$-avoiding and represents the complete graph since each letter appears once, thereby alternating with every other letter.
\end{proof}

\begin{theorem}\label{path123}
Paths are $123$-representable.
\end{theorem}
\begin{proof}
A path on $n$ vertices can be represented by the $123$-avoiding word $n(n-1)n(n-2)(n-1)(n-3)(n-1) \ldots 23121$. This is clearly $123$-avoiding and represents a path since every vertex alternates only with the one before it and after it, with the exception of $1$ and $n$, which alternate only with $2$ and $n-1$, respectively.
\end{proof}

\begin{theorem}
Cycles are $123$-representable.
\end{theorem}
\begin{proof}
A cycle on $n$ vertices can be represented by the $123$-avoiding word $(n-1)n(n-2)(n-1)(n-3)(n-1) \ldots 2312$. This is the same word as the one which represents a path, except with the first $n$ and last $1$ deleted to make it alternate with $1$ as well as with $n-1$.
\end{proof}

\section{Graphs represented by pattern avoiding $2$-uniform words}\label{2unifgraphs}

In this section we provide some properties of graphs which are representable by either $123$-avoiding or $132$-avoiding $2$-uniform words. As stated in Lemma \ref{2uniflemma}, the disjoint union of graphs which can be represented by $132$-avoiding $2$-uniform words is a $132$-representable graph as well. Here, we prove a more general result.

\begin{theorem}\label{onlyonenot2unif}
Let $G_1, G_2, \dots , G_k$ be connected $132(123)$-representable components of a graph $G$. Then $G$ is $132(123)$-representable if and only if at most one of the connected components cannot be $132(123)$-represented by a $2$-uniform word.
\end{theorem}
\begin{proof}

First we show that $G$ is not $132$($123$)-representable if at least $2$ of the components are not $132$($123$)-representable by a $2$-uniform word. We first prove this for $132$-representable graphs.

Assume for the sake of contradiction that $G$ is $132$-representable and can be represented by some $w$. Let components $G_i$ and $G_j$ be graphs not representable by a $132$-avoiding $2$-uniform word. By Theorem \ref{shitcase} we can assume that $w$ has at most $2$ copies of each letter.

It is clear that $w$ must have subwords $w_i$ and $w_j$, formed by removing all letters from $w$ that do not appear in $G_i$ and $G_j$, respectively, which represent both $G_i$ and $G_j$ (otherwise it would be impossible for them to be components of $G$). It is impossible for either $w_i$ or $w_j$ not to be $132$-avoiding, thus $G_i$ and $G_j$ must both be $132$-representable. Now, if $w_i$ is not $2$-uniform, then some letter $a$ must appear only once (as no letter appears more than two times). Similarly, some $b$ must appear only once in $w_j$. However, this would mean that $a$ and $b$ alternate, and then $G_i$ and $G_j$ are not disconnected components of $G$. We have reached a contradiction, thus we are done.

Next we show that $G$ is $132$-representable if it has at most one component that is not representable by a $132$-avoiding $2$-uniform word.

First we relabel the components the following way: $G_1$ has vertices $1, 2, \dots ,t_1$; $G_2$ has vertices $t_1+1, t_1 + 2, \dots ,t_2$; $\dots$; $G_k$ has vertices $t_{k-1} + 1, \dots ,t_k$. Now let $w_1, w_2, \dots ,w_k$ be $132$-representants for $G_1, G_2, \dots, G_k$, respectively. Now we show that the word $w = w_kw_{k-1}\cdots w_1$ is a $132$-representant for $G$. It is easy to see that $w$ is $132$-representable, as no $w_i$ contains a $132$-pattern, and all letters appearing after $w_i$ are smaller than all letters appearing in $w_i$. Furthermore, $w$ represents $G$ since it is impossible for any letter appearing twice in $w$ to alternate with anything that is not in its component's word. Since at most one of the $w_i$ can have letters appearing fewer than two times, they will not alternate with any letters in any other $w_j$. Thus $G$ is $132$-representable.

\end{proof}
The proof for $123$-representable graphs is exactly the same.

Note that this gives us a new class of $123$- and $132$-representable graphs, namely those formed by the disjoint union of several graphs representable by a $2$-uniform word and one that possibly is not.

\subsection{Graphs represented by $123$-avoiding $2$-uniform words}
Here we prove that cycles and complete graphs can be $123$-represented by $2$-uniform words.

\begin{theorem}
Any path is $123$-representable by a $2$-uniform word.
\end{theorem}

\begin{proof}
This follows directly from the construction given in Theorem \ref{path123}. 
\end{proof}

\begin{theorem}
Any complete graph is $123$-representable by a $2$-uniform word.
\end{theorem}
\begin{proof}
A complete graph on $n$ vertices can be represented by the $2$-uniform word $n(n-1)(n-2) \cdots 1n(n-1) \cdots 1$.
\end{proof}

\subsection{Graphs represented by $132$-avoiding $2$-uniform words}

\begin{theorem}
Any tree is $132$-representable by a $2$-uniform word.
\end{theorem}
\begin{proof}
The recursive algorithm provided in \cite{mainpaper} gives a word which has one copy of the root vertex and two copies of every other vertex. The recursively generated word is $w(T_r)w(T_{r-1}) \cdots w(T_1)1n_1n_2 \cdots n_r$, where $1$ is the root vertex, $n_1, n_2, \dots , n_r$ are the children of the root vertex, and $w(T_m)$ denotes the word generated in the same way, but representing the subtree with $n_m$ as the root vertices. We claim that the word $w(T_r)w(T_{r-1}) \cdots w(T_1)1n_1 \linebreak n_2 \cdots n_r1$ is also $132$-avoiding and represents the same tree. It is easy to see that it is still $132$-avoiding since adding $1$ at the end cannot possibly form a $132$-pattern, Furthermore, $1$ is still alternating with all of $n_i$, and not with any other vertices. Thus any tree can be represented by a $2$-uniform $132$-avoiding word.
\end{proof}

\begin{theorem}\label{complete not 2-unif}
The complete graph $K_n$, with $n>3$, is not $132$-representable by a $2$-uniform word.
\end{theorem}
\begin{proof}
We prove this by contradiction. Assume there is a $132$-avoiding $2$-uniform word $w$ that represents a complete graph on at least $4$ vertices. Since $1$ must appear twice in $w$, we have $w = w_11w_21w_3$, with $2$, $3$, and $4$ appearing in $w_2$. It is easy to see that they must appear in increasing order, since otherwise there will be a $132$-pattern. Each of $2$, $3$, and $4$ must also appear either in $w_1$ or in $w_3$. Both $2$ and $3$ cannot appear in $w_3$ since otherwise they form a $132$-pattern with the $1$ and $4$ in $w_2$. Thus they appear in $w_1$ in increasing order because otherwise the $2$ and $3$ would not be alternating. Finally, the $4$ cannot appear in $w_3$ because then it would not be alternating with $2$ or with $3$. If the $4$ is in $w_1$ it must come before the $2$, since otherwise it would form a $132$-pattern with the $2$ in $w_1$ and the $3$ in $w_2$. However, this makes it impossible for the $4$ to alternate with $2$ and $3$, a contradiction. This completes the proof.
\end{proof}

\section{$132$-representable graphs}\label{132graphs}
Our final result in this paper answers a question posed in \cite{mainpaper}. 

\begin{theorem}
Not all circle graphs are $132$-representable.
\end{theorem}
\begin{proof}
As shown in Theorem \ref{complete not 2-unif}, $K_4$ is not representable by a $2$-uniform $132$-avoiding word. Then, by Theorem \ref{onlyonenot2unif}, the disjoint union of two complete graphs of size $4$ (Figure \ref{disjoint K4}) is not $132$-representable. However, Figure \ref{circlek4} demonstrates that it is a circle graph. Thus not all circle graphs are $132$-representable.
\end{proof}

\begin{figure}[htbp]
\centering
\includegraphics[width=0.4\textwidth]{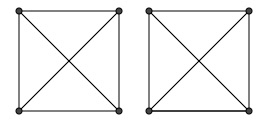}
\caption{\label{disjoint K4}The disjoint union of two complete graphs of size $4$.}
\end{figure}

\begin{figure}[htbp]
\centering
\includegraphics[width=0.4\textwidth]{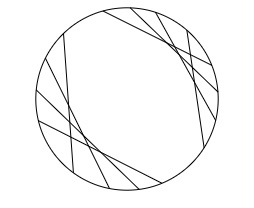}
\caption{\label{circlek4}The circle representation of the disjoint union of two complete graphs of size $4$, which demonstrates that this is indeed a circle graph.}
\end{figure}

One might wonder if all circle graphs are either $123$- or $132$-representable. This is not true, as can be seen in the simple counterexample in Figure \ref{counter}. It is not $123$-representable for the same reason that a star on $7$ vertices is not, and it is not $132$-representable for the same reason that the disjoint union of two complete graphs greater than $K_3$ is not $132$-representable.

\begin{figure}[htbp]
\centering
\includegraphics[width=0.6\textwidth]{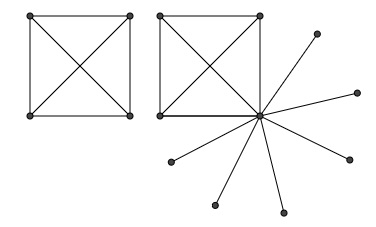}
\caption{\label{counter}An example of a circle graph which is neither $132$-representable nor $123$-representable.}
\end{figure}

\section{Open research directions}\label{opendir}

A natural next step in the study of pattern-representable graphs would be to investigate longer patterns, or to find more examples of $132$/$123$-representable and non-representable graphs.

In particular, the following questions may be of interest.

\begin{description}
\item[Question 1:]{Is the disjoint union of two complete graphs of size $4$ the smallest non-$132$-representable circle graph?}

\item[Question 2:]{Is the star on $7$ vertices the smallest non-$123$-representable circle graph?}

\end{description}

\section{Acknowledgments}

This research was conducted at the University of Minnesota Duluth REU program, supported by NSF grant 1358659 and NSA grant H98230-16-1-0026. I would like to thank Joe Gallian for the wonderful environment for research at UMD, for suggesting the problem, and for his constant encouragement and constructive critique of the manuscript. I would also like to thank Matthew Brennan and David Moulton for reading my paper and for greatly helpful discussions about circle graphs.

\bibliographystyle{plain}

\end{document}